\documentclass[11pt]{amsart}
\usepackage{pstricks,pst-plot}
\usepackage{mathrsfs}

\definecolor{Black}{cmyk}{0,0,0,1}
\definecolor{OrangeRed}{cmyk}{0,0.6,1,0}            
\definecolor{DarkBlue}{cmyk}{1,1,0,0.20}
\definecolor{myblue}{rgb}{0.66,0.78,1.00}
\definecolor{Violet}{cmyk}{0.79,0.88,0,0}
\definecolor{Lavender}{cmyk}{0,0.48,0,0}

\parskip=\smallskipamount

\newtheorem{theorem}{Theorem}[section]

\newtheorem{lemma}[theorem]{Lemma}
\newtheorem{corollary}[theorem]{Corollary}
\newtheorem{proposition}[theorem]{Proposition}

\theoremstyle{definition}

\newtheorem{example}[theorem]{Example}

\newtheorem{remark}[theorem]{Remark}

\newcommand{\cC}{\mathcal{C}}

\newcommand{\cG}{\mathcal{G}}

\newcommand{\C}{\mathbb{C}}
\newcommand{\D}{\mathbb{D}}

\newcommand{\N}{\mathbb{N}}

\newcommand{\hol}{\mathscr{O}}
\newcommand{\debar}{\bar{\partial}}

\numberwithin{equation}{section}

\begin{document}
\title[Uniform Algebras]{Uniform Algebras and Approximation on Manifolds}
\author{H\aa kan Samuelsson}
\author{Erlend Forn\ae ss Wold}
\address{Matematisk Institutt, Universitetet i Oslo,
Postboks 1053 Blindern, 0316 Oslo, Norway}
\email{erlendfw@math.uio.no}
\email{haakansa@math.uio.no}

%
%
\subjclass[2000]{}
\date{\today}

\subjclass[2000]{}
\date{\today}


\maketitle

\begin{abstract}
Let $\Omega \subset \mathbb{C}^n$ be a bounded domain and let $\mathcal{A} \subset \mathcal{C}(\overline{\Omega})$ be a uniform 
algebra generated by a set $F$ of holomorphic and pluriharmonic functions. Under natural assumptions on $\Omega$ and $F$ 
we show that the only obstruction to $\mathcal{A} = \mathcal{C}(\overline{\Omega})$ is that there is a holomorphic disk 
$D \subset \overline{\Omega}$ such that all functions in $F$ are holomorphic on $D$, i.e., the only obstruction is the obvious 
one. This generalizes work by A. Izzo. We also have a generalization of Wermer's maximality theorem to the (distinguished 
boundary of the) bidisk.
\end{abstract}

\section{introduction}

In this article we will discuss versions of two theorems of 
John Wermer: His well known maximality theorem \cite{Wermer53} states that 
if $f\in\mathcal C(b\mathbb D)$, then either $f$ is the boundary value of a holomorphic
function on the disk, or $[z,f]_{b\mathbb D}=\mathcal C(b\mathbb D)$. Here
$[z_1,f]_{b{\mathbb D}}$ denotes the uniform algebra generated by $z$ and $f$ on the boundary of the disk.
A closely related result is the following \cite{Wermer65}:
Let $f\in\mathcal C^1(\overline{\mathbb D})$ and 
assume that the graph $\mathcal G_f(\overline{\mathbb D})$ of $f$ over $\overline{\mathbb D}$ in $\mathbb C^2$
is polynomially convex.  Let $S:=\{z\in\overline{\mathbb D}:\overline\partial f(z)=0\}$.   Then 
$$
[z_1,f]_{\overline{\mathbb D}}=\{g\in\mathcal C({\overline{\mathbb D}}):g|_S\in\mathcal O(S)\}, 
$$
Note that if $f$ is harmonic, then $f$ is holomorphic or $\mathcal O(S)=\mathcal C(S)$. \

We let $PH(\Omega)$ denote the pluriharmonic functions on a domain $\Omega\subset\mathbb C^n$, 
and we let $\Gamma^2$ denote the distinguished boundary of the bidisk $\D^2$. \

Our most complete results are in $\mathbb C^2$ and are contained in the following two theorems:

\begin{theorem}\label{polydisk}
Let $h_j\in PH(\mathbb D^2)\cap\mathcal C^1(\overline{\mathbb D^2})$ for $j=1,...,N$.
Then either there exists a holomorphic disk in $\overline{\mathbb D^2}$ where 
all $h_j$-s are holomorphic, or $\mathbb [z_1,z_2,h_1,...,h_N]_{\overline{\mathbb D^2}}=\mathcal C(\overline{\mathbb D^2})$.
\end{theorem}

\begin{theorem}\label{maximality}
Let $f_j\in\mathcal C(\Gamma^2), j=1,...,N$, $N\geq 1$, and 
assume that each $f_j$ extends to a pluriharmonic function on $\mathbb D^2$.  
Then either 
$$
[z_1,z_2,f_1,...,f_N]_{\Gamma^2}=\mathcal C(\Gamma^2),
$$ 
or there exists a nontrivial closed algebraic variety $Z\subset\overline{\mathbb D^2}\setminus\Gamma^2$ 
with $\overline Z\setminus Z\subset\Gamma^2$, and the pluriharmonic extensions
of the $f_j$-s are holomorphic on $Z$.  
\end{theorem}

(We will give an example (Example \ref{example}) with $N=1$, where no such variety can exist, \emph{i.e.}, 
$\mathcal C(\Gamma^2)=[z_1,z_2,f]$.) \

\begin{theorem}\label{domain}
Let $\Omega\subset\mathbb C^n$ be a polynomially convex $\mathcal C^1$-smooth domain.  Let 
$h_j\in PH(\Omega)\cap\mathcal C(\overline\Omega)$ for $j=1,...,N$, 
and assume that $\mathbb [z_1,...,z_n,h_1,...,h_N]_{b\Omega}=\mathcal C(b\Omega)$.  
Then either there exists a holomorphic disk in $\Omega$ where 
all $h_j$-s are holomorphic, or $[z_1,...,z_n,h_1,...,h_N]_{\overline{\Omega}}=\mathcal C(\overline\Omega)$.
\end{theorem}

We remark that if $\Omega$ was strictly pseudoconvex, we would get the same result with the 
algebra $\mathcal A(\overline\Omega)[h_1,...,h_N]$ instead.  \

In the one-variable case, Theorem \ref{polydisk} is due to  \v{C}irca \cite{Circa} (see also 
Axler-Shields \cite{AxlerShields}). 
Theorems \ref{polydisk} and \ref{domain}  generalize work by A. Izzo \cite{Izzo1993} and Weinstock \cite{Weinstock2}.
Theorem \ref{maximality}
generalizes Wermer's maximality theorem to $\mathbb C^2$ to the effect that  analyticity 
is the only obstruction to the full algebra being generated.  \

Assume in addition to the conditions in Theorem \ref{domain} that $\Omega$ is strictly pseudoconvex and
that  $h_j\in\mathcal C^1(\overline\Omega)$ 
for $j=1,...,N$, and define 
$$
S:=\{z\in\overline\Omega: \overline\partial h_{i_1}\wedge\cdot\cdot\cdot\wedge\overline\partial h_{i_n}(z)=0
\mbox{ for all } 1\leq i_j\leq N\}.
$$
Izzo's result is that if $S\cap\Omega=\emptyset$ and $N=n$ then $[z_1,...,z_n,h_1,...,h_n]_{\overline{\Omega}}=\mathcal C(\overline\Omega)$.  In this case $S\subset b\Omega$ is a peak interpolation set by Weinstock \cite{Weinstock}, and 
by the assumption on the wedge products we have that $\mathcal G_h(\overline\Omega\setminus S)$
is totally real.  The pluriharmonicity of the $h_j$-s guaranties that $\mathcal{G}_h(\overline\Omega)$
is polynomially convex, hence the conclusion follows from Theorem \ref{approx} below. \

According to Theorem \ref{domain} there  is no need in general to assume that $S\cap\Omega=\emptyset$.  
For a generic choice of $h_j$-s (as long as $N\geq n$) we have that $S\cap\Omega\neq\emptyset$
will not prevent the full algebra from being generated, but in exceptional cases the presence 
of a nontrivial analytic set $Z\subset\Omega$ with all $h_j$-s analytic along $Z$ will be an obstruction. 
We have computational criteria for detecting such a set.  For approximation 
of continuous functions on $\mathcal G_h(\overline\Omega)$ it is then 
necessary and sufficient to assume that the function sought approximating
is approximable on $\mathcal G_h(Z\cup b\Omega)$ (as above, the boundary might 
be covered by other existing results).  \

\section{Proof of a Theorem of A. Izzo}

Recently A. Izzo \cite{Izzo2010} proved a conjecture of Freeman \cite{Freeman}
regarding uniform algebras on manifolds. 
At the core of our proofs of the theorems in our introduction is an approximation result due to P. E. Manne \cite{Manne} 
concerning $\mathcal C^1$-approximation by holomorphic functions on totally 
real sets attached to compact sets (cf. \cite{ManneWoldOvrelid}). 
We will demonstrate its strength by giving 
a very short proof of the result of Izzo.  Whereas 
Izzo uses the Arens-Calder\'{o}n lemma to utilize techniques of Weinstock \cite{Weinstock2}
(cf. Berndtsson \cite{Berndtsson}), we will use it to utilize techniques of Manne \cite{Manne}
(cf. Manne, Wold and \O vrelid \cite{ManneWoldOvrelid}).

\begin{theorem}(A. Izzo)\label{izzo}
Let $M$ be an $m$-dimensional $\mathcal C^1$-manifold-with-boundary, 
and let $X$ be a compact subset of $M$.  Let $A$ be 
a uniform algebra on $X$ generated by a family $\Phi$
of complex valued functions $\mathcal C^1$ on $M$, assume that 
the maximal ideal space of $A$ is $X$, and let
$$
E:=\{p\in X:df_1\wedge\cdot\cdot\cdot\wedge df_m(p)=0 \mbox{ for all } f_1,..,f_m\in\Phi\}.
$$
Then $A=\{g\in\mathcal C(X): g|_E\in A|_E\}$.
\end{theorem}

\begin{proof}
It is enough to show that for any point $x\notin E$ there 
exists an open neighborhood $D$ of $x$ such 
that any continuous function on $X$ with compact support in $D$
is in $A$.  By definition(see Lemma \ref{totreellLemma}) there exist $f_1,...,f_n\in A$
and a neighborhood $D'$ of $x$ such that $F(D'\cap X)$
is a totally real set, where $F=(f_1,...,f_n)$.   
Since $A$ needs to separate points of $X$  (since 
the maximal ideal space of $A$ is $X$) we may, 
by possibly having to add more functions, assume that 
$F^{-1}(F(X)\cap F(D'\cap X))=D'\cap X$.  Let 
$\Omega$ be a neighborhood of $X_0=F(X)$ as 
in Lemma \ref{projection}.   Since the maximal 
ideal space of $A$ is $X$ it follows from the 
Arens-Calder\'{o}n lemma (see \emph{e.g.} \cite{Gamelin})
that there exist functions $f_{n+1},...,f_{n+m}\in A$
such that, writing $\tilde F=(f_1,...,f_n,f_{n+1},...,f_{n+m})$, 
we have that $\pi_n(\widehat{\tilde F(X)})\subset\Omega$.
The result now follows from Lemma \ref{projection}.
\end{proof}

\begin{lemma}\label{projection}
Let $X_0\subset\mathbb C^n$ be a compact set and assume that 
$z_0\in X_0$ is a totally real point, \emph{i.e.}, $X_0$ is a 
totally real set near $z_0$.  Then there exists a neighborhood 
$\Omega$ of $X_0$ and a neighborhood $U'$ of 
$z_0$ such that  the following holds: Let $\pi_n:\mathbb C^n\times\mathbb C^m\rightarrow\mathbb C^n$
denote the projection onto the first $n$ coordinates, let $X_1\subset\mathbb C^{n+m}$
be compact with $\pi_n(X_1)=X_0$, and assume that $\pi_n(\widehat X_1)\subset\Omega$. 
Then for any $f\in\mathcal C_0(U'\cap X_0)$ we have that
$$
(f\circ\pi_n)|_{X_1}\in [z_1,...,z_{n+m}]_{X_1}.
$$
\end{lemma}

\begin{proof}
We follow the proof of Proposition 3.10 in \cite{ManneWoldOvrelid}. 
Let $V$ be a neighborhood of $z_0$ such that $V\cap X_0$
is totally real.  
The following is the content of the proposition on page 522 in \cite{Manne}: There are neighborhoods
$$
U'\subset\subset U''\subset\subset U\subset\subset V
$$
of $z_0$, and a neighborhood $W\subset U$ of $X_0\cap bU''$ such that 
if $f\in\mathcal C(X_0\cap V)$ has compact support in $U'$ there 
exists a sequence of holomorphic functions $h_j\in\mathcal O(U)$
such that $\|h_j-f\|_{X_0\cap U}\rightarrow 0$ and $\|h_j\|_{W}\rightarrow 0$
uniformly as $j\rightarrow\infty$.  Let $\{\Omega_j\}$ 
be a neighborhood basis of $X_0$ and 
define 
$U_j^1=\Omega_j\cap U''$, $U_j^2=(\Omega_j\setminus U'')\cup (W\cap\Omega_j)$. 
If $j$ is large enough, we have that $U_j^2$ is 
open, and $\Omega_j=U_j^1\cup U_j^2$ and $U_j^1\cap U_j^2\subset W$.  
Fix a j large enough so that this holds and drop the subscript $j$.  \

Assume that $\pi_n(\widehat X_1)\subset\Omega$, and choose a Runge and 
Stein neighborhood $\tilde\Omega$ of $X_1$ with $\pi_n(\tilde\Omega)\subset\Omega$.
We solve a Cousin problem on $\tilde\Omega$ with respect
to the cover $\tilde U^1=\tilde\Omega\cap\pi_n^{-1}(U^1)$
and $\tilde U^2=\tilde\Omega\cap\pi_n^{-1}(U^2)$.  \

Let $\tilde h_j=h_j\circ\pi_n$.  By the solution of Cousin I with estimates 
there exist sequences $g_j^i\in\mathcal O(\tilde U^i)$ such that 
$\tilde h_j=g_j^1-g_j^2$ on $\tilde U^1\cap\tilde U^2$, and 
such that $g_j^i\rightarrow 0$ uniformly on compact subsets of $\tilde U^i$.  
The sequence $g_j$ defined as $h_j-g_j^1$ on $\tilde U^1$
and $-g_j^2$ on $\tilde U^2$ will converge to 
$f\circ\pi_n|_{X_1}$.  
\end{proof}

\begin{remark}
We remark that Izzo's somewhat more general Theorem 1.2. can be proved 
in a similar fashion.  
\end{remark}

\section{Preliminaries}

\subsection{Toally real manifolds}

\begin{lemma}\label{totreellLemma}
Let $M$ be a smooth manifold of real dimension $m$ and let $f=(f_1,\ldots,f_N)\colon M\to \C^N$
be a $\cC^1$-smooth map. If $x\in M$ and if there are $f_{i_1},\ldots,f_{i_m}$ such that 
\begin{equation*}
df_{i_1}\wedge \cdots \wedge df_{i_m}(x)\neq 0,
\end{equation*}
then the map $f$ is an embedding near $x$ and the image $T_{f(x)}f(M)$ is totally real. 
\end{lemma}

\begin{proof}
The first assertion is clear (actually the condition that the wedge product is non-vanishing is stronger than
the condition that $f$ is an embedding near $x$).
This means that the 
$\C$-linear map $df \colon T^{\C}_x M \to \C^N$ has maximal rank, where $T^{\C}_x M$ is the 
complexified tangent space of $M$ at $x$.
Hence, $T_{f(x)} f(M)\cap JT_{f(x)} f(M)=df (T_x M \cap JT_x M)=0$.
\end{proof}

\begin{corollary}\label{dbarCor}
Let $X$ be a complex manifold of dimension $n$, let $h=(h_1,\ldots,h_N)\colon X\to \C^N$ be a $\mathcal C^1$-smooth map,
and let $\cG_h(X)\subset X\times \C^N$ be the graph of $h$. If $x\in X$ and there are $h_{i_1},\ldots,h_{i_n}$ such that
\begin{equation*}
\debar h_{i_1}\wedge \cdots \wedge \debar h_{i_n} (x)\neq 0,
\end{equation*}
then the tangent space of $\cG_h(X)$ at $(x,h(x))$ is totally real.
\end{corollary}

\begin{proof}
Consider the map $X\to X\times \C^N$ defined by $z\mapsto (z,h(z))$
and denote by $dz$ the form $dz_1\wedge \cdots \wedge dz_n$ where $z_j$ are local coordinates centered at $x$. 
For bidegree reasons
\begin{equation*}
dz \wedge d h_{i_1}\wedge \cdots \wedge d h_{i_n}=
dz \wedge \debar h_{i_1}\wedge \cdots \wedge \debar h_{i_n},
\end{equation*}
and by assumption the latter product is non-vanishing at $x$.
\end{proof}

\subsection{Pluriharmonic functions}\label{pharm}

Let $h=(h_1,\ldots,h_N)\colon \Omega \to \C^N$, $n\leq N$, be a pluriharmonic mapping.
Let $e_1,\ldots,e_N$ be a basis for $\C^N$. Consider $\{e_j\}$ as a frame for  
the trivial $\mbox{rank}\, N$-bundle $E\to \Omega$ and consider $h=h_1e_1+\cdots+h_Ne_N$ as a section of $E$. We let 
$H_1$ be the section of $T^*_{0,1}(\Omega)\wedge E\simeq T^*_{0,1}(\Omega)\otimes E$ defined by
\begin{equation}\label{H1}
H_1=\debar h=\debar h_1\wedge e_1+\cdots +\debar h_N\wedge e_N
\end{equation}
and we define $H_k$ as sections of $T^*_{0,k}(\Omega)\wedge \Lambda^k E$ by
\begin{equation}\label{Hk}
H_k=(H_1)^k/k!=H_1\wedge \cdots \wedge H_1/k!=\pm \sum'_{|I|=k} \debar h_I \wedge e_I,
\end{equation}
where $\sum'_I$ means that we sum over increasing multiindices
and $\debar h_I=\debar h_{I_1}\wedge \cdots \wedge \debar h_{I_k}$.
Since the $H_k$ are invariantly defined, this construction also makes sense for pluriharmonic mappings
from complex manifolds $X\to \C^N$. Moreover, $H_k$ is anti-holomorphic, or more formally, an 
anti-holomorphic $(0,k)$-form with values in $\Lambda^k E$. If $i\colon Y\hookrightarrow X$ is a $k$-dimensional 
complex submanifold, we will write $Y_H^k$ for the set of points of $Y$ where $i^*H_k$ vanishes;
if $\zeta_1,\ldots,\zeta_k$ are local coordinates on $Y$ this set coincides with the common zero set of the 
tuple of all $k\times k$-subdeterminants of the matrix
\begin{equation*}
\left( \begin{array}{ccc}
       \partial h_1/\partial \bar{\zeta}_1 & \cdots & \partial h_1/\partial \bar{\zeta}_k \\
       \vdots & & \vdots \\
       \partial h_N/\partial \bar{\zeta}_1 & \cdots & \partial h_N/\partial \bar{\zeta}_k \end{array}\right).
\end{equation*}
Hence, $Y^k_H$ is an analytic subset of $Y$.

\begin{proposition}\label{holodisk}
Let $X$ be a complex manifold of dimension $n$ and let $h=(h_1,\ldots,h_N) \colon X\to \C^N$, $N\geq n$, be a pluriharmonic mapping.
If there is a (germ of a) complex submanifold $Y\subset X$ of dimension $k$ such that 
$Y^k_H=Y$, then there is a (germ of a) holomorphic disk in $Y$ where all $h_j$ are holomorphic.  
\end{proposition}

\begin{proof}
It suffices to show that if $H_n$ vanishes identically, then $X$ contains a holomorphic disk where all $h_j$ are holomorphic.
We prove this by induction over $n$. If $n=1$, then the condition that $H_1$ vanishes identically precisely means that
all $h_j$ are holomorphic. Assume that the statement is true for all $n\leq k$ and all $N\geq n$. Let $X$ be $k+1$-dimensional
and assume that $h\colon X\to \C^N$, $N\geq k+1$, is pluriharmonic and that $H_{k+1}$ vanishes identically. 
We may assume that there is some $h_j$ which is not holomorphic on $X$. Assume for simplicity that $h_1$ is not holomorphic
and let $x\in X$ be a point such that $\debar h_1(x)\neq 0$.
In a neighborhood of $x$ we may write
\begin{equation*}
h_1=\mathfrak{Re}\, g+f,
\end{equation*}
where $g$ and $f$ are holomorphic. Since $\debar h_1(x)\neq 0$ it follows that $d g(x)\neq 0$. Let 
$Y=\{z;\, g(z)=g(x)\}$ be the level set of $g$ through $x$. Then $i\colon Y\hookrightarrow X$ is a complex 
$k$-dimensional submanifold through $x$. Choose local coordinates $(w_0,w_1,\ldots,w_k)=(w_0,w')$ centered at 
$x$ such that $Y=\{w_0=0\}$. Since $g$ is constant on $Y$ we have $i^*\debar h_1=i^*\overline\partial\mathfrak{Re}\, g=0$, i.e.,
$h_1$ is holomorphic on $Y$. Thus, since $\debar h_1(x)\neq 0$ it follows that 
\begin{equation*}
\frac{\partial h_1}{\partial \bar{w}_0}(x)\neq 0, \quad \frac{\partial h_1}{\partial \bar{w}_j}=0, \,\, j=1,\ldots,k,
 \,\,\, \mbox{on}\,\, Y.
\end{equation*}
Now, let $I\subset \{1,\ldots,N\}$, $|I|=k$. If $1\in I$, then it is obvious that 
$i^* \debar h_I=i^* \debar h_{I_1}\wedge \cdots \wedge \debar h_{I_k}=0$ so we may assume that 
$1\notin I$. Since $H_{k+1}$ vanishes identically we have
\begin{equation*}
0=\det \left( \begin{array}{cc}
              \partial h_1/\partial \bar{w}_0 & 0 \\
                * &  \partial h_I/\partial \bar{w}' \end{array} \right)
= \frac{\partial h_1}{\partial \bar{w}_0}\cdot \det \left( \frac{\partial h_I}{\partial \bar{w}'}\right)
\end{equation*}
on $Y$ and since $\partial h_1/\partial \bar{w}_0\neq 0$ close to $x$ it follows that
$i^* \debar h_I=0$. Hence, $i^*H_k$ vanishes identically in a neighborhood in $Y$ of $x$ and it follows from the 
induction hypothesis that $Y$ contains a holomorphic disk where all $h_j$ are holomorphic.
\end{proof}

\subsection{Polynomial convexity and approximation on stratified totally real sets}

We now consider approximation on stratified totally real sets.
The following result gives a sufficient condition for when a compact polynomially convex set $X\subset\mathbb C^n$ has the property
that $\mathcal{C}(X)=[z_1,...,z_n]_X$. The technical and main 
part of the proof is contained in \cite[Proposition~3.13]{ManneWoldOvrelid}. For convenience
of the reader we state here a simplified version of this proposition: 

\begin{proposition}\cite[Proposition~3.13]{ManneWoldOvrelid}\label{P1}
Let $K\subset\mathbb C^n$ be a compact set, $M\subset\mathbb C^n$ a totally real set, 
$M_0\subset M$ compact, and assume that $K\cup M_0$ is polynomially convex.  
Then for any $f\in\mathcal C(K\cup M_0)$ with $Supp(f)\cap K=\emptyset$, there 
exists a sequence $\{h_j\}_{j=1}^\infty\subset\mathcal O(\mathbb C^n)$ such that 
$\|h_j-f\|_{K\cup M_0}\rightarrow 0$ as $j\rightarrow\infty$.
\end{proposition}



\begin{theorem}\label{approx}
Let $X$ be a polynomially convex compact set in $\C^n$ and assume that there are closed sets 
$X_0\subset \cdots \subset X_N=X$ such that 
\begin{itemize}
\item[(i)] $X_j\setminus X_{j-1}$, $j=1,\ldots,N$, is a totally real set.
\end{itemize}
Then for $f\in\mathcal C(X)\cap\hol(X_0)$ we have that $f\in [z_1,...,z_n]_X$.
In particular, if $\mathcal{C}(X_0)=\hol(X_0)$ then $\mathcal C(X)=[z_1,...,z_n]_X$.
\end{theorem}

\begin{proof}
We notice that each $X_j$ has to be polynomially convex. In fact, we trivially have 
$\widehat{X_{N-1}}\subseteq \widehat{X_N}=X_N$. 
Moreover, if $x\in X_N\setminus X_{N-1}$ then, by Proposition \ref{P1}, there is a 
polynomial $P_x$ such that $\|P_x\|_{X_{N-1}} < |P_x(x)|$. Hence, $x\notin \widehat{X_{N-1}}$ and so
$\widehat{X_{N-1}}=X_{N-1}$. Repeating the argument we see that $X_{N-2}$ is polynomially convex, and so on.

Now let $f\in\mathcal C(X)\cap\hol(X_0)$.  Proceeding by induction we will show that 
if $f\in\hol (X_k)$ then $f\in\hol (X_{k+1})$ for $k\geq 0$.   Let $\epsilon>0$, and let 
$g\in\hol (\mathbb C^n)$ with $\|g-f\|_{X_k}<\frac{\epsilon}{4}$.  Let $U$ be an open neighborhood 
of $X_k$ such that $\|g-f\|_{\overline U\cap X_{k+1}}<\frac{\epsilon}{2}$.   
Let $\chi\in\mathcal C^\infty_0(U)$ with $0\leq\chi\leq 1$ and $\chi\equiv 1$ near $X_k$.  Then 
$$
g + (1-\chi)\cdot (f-g)
$$
is an $\epsilon$-approximation of $f$ on $X_{k+1}$, $g$ is entire, and $(1-\chi)\cdot(f-g)$ 
is  approximable on $X_{k+1}$ by entire functions by Proposition \ref{P1}.

\end{proof}

\begin{lemma}\label{pcgraph}
Let $\Omega\subset\mathbb C^n$ be a $\mathcal C^1$-smooth polynomially 
convex domain, 
and let $h_j\in PH(\Omega)\cap\mathcal C(\overline\Omega)$ for 
$j=1,...,N$.  Then $\cG_h(\overline\Omega)$
is polynomially convex. 
\end{lemma}

\begin{proof}
Let $(z_0,w_0)\in (\C^n\times \C^N)\setminus \cG_h(\overline{\Omega})$. If $z_0 \notin \overline{\Omega}$ it is clear that 
$(z_0,w_0)\notin \widehat{\cG_h(\overline{\Omega})}$ since $\overline{\Omega}$ is polynomially convex. Assume that 
$z_0\in \overline{\Omega}$. For some $j$ we have $\mathfrak{Re}\, (w_{0,j}-h_j(z_0))\neq 0$
and we may assume that $g(z,w):=\mathfrak{Re}\, (w_{j}-h_j(z))$ is positive at $(z_0,w_0)$.
Then $g$ is pluriharmonic in $\Omega$, continuous up to the boundary, and satisfies
$g(z_0,w_0)> \|g\|_{\cG_h(\overline\Omega)}=0$. By \cite[Theorem 1]{FW}, the function $g$ is uniformly approximable
on $\overline{\Omega}$ by functions in $PSH(\Omega)\cap \mathcal{C}^{\infty}(\overline{\Omega})$.
Thus, since in addition $\overline{\Omega}$ is polynomially convex, 
it follows that there is an open Runge and Stein neighborhood $\widetilde{\Omega}\supset \overline{\Omega}$ 
and a function $\tilde{g}\in PSH(\widetilde{\Omega}\times \C^N)$ such that $\tilde{g}(z_0,w_0)> \|\tilde{g}\|_{\cG_h(\overline\Omega)}$.
It thus follows that $(z_0,w_0)\notin \widehat{\cG_h(\overline\Omega)}_{\hol(\widetilde{\Omega}\times \C^N)}$ and so
$(z_0,w_0)\notin \widehat{\cG_h(\overline\Omega)}$.
\end{proof}

The first part of the following result can be found in \cite{ManneWoldOvrelid}. 

\begin{proposition}\label{easybishop}
Let $K\subset\mathbb C^N$ be compact, let $F:\mathbb C^N\rightarrow \mathbb C^M$
be the uniform limit on $K$ of entire functions, and let $Y=F(K)$; note 
that $F$ extends to $\widehat K$.  For a point $y\in Y$, 
let $F_y$ denote the fiber $F^{-1}(y)\subset\widehat K$, and let $K_y$ denote 
the restricted fiber $F_y\cap K$.   The following holds: 
\begin{itemize}
\item[i)] if $y$ is a peak point for the algebra $[z_1,...z_M]_Y$, then $\widehat K\cap F_y=\widehat K_y$, and 
\item[ii)] if $[z_1,...,z_M]_{Y}=\mathcal C(Y)$ then 
$$
[z_1,...,z_N]_K=\{f\in\mathcal C(K):f|_{K_y}\in [z_1,...,z_N]_{K_y} \mbox{ for all } y\in Y\}.
$$
\end{itemize}
\end{proposition}
\begin{remark}
Note that ii) is a consequence of Bishop's antisymmetric decomposition theorem.  
The proof we will give here is due to Nils \O vrelid, and is almost trivial. 
\end{remark}

\begin{proof}
For the proof of i) see Proposition 4.3 in \cite{ManneWoldOvrelid}.
To prove ii) simply glue together functions which are good 
near the fibers $K_y$ using a continuous partition of 
unity on the projection $Y=F(K)$.  
\end{proof}

\section{Proof of Theorem \ref{polydisk}}\label{bidisksektion}
We first note that the graph $\mathcal{G}_h(\overline{\D}^2)$ is polynomially convex. The proof of Lemma~\ref{pcgraph}
goes through except for that we, strictly speaking, cannot use Theorem~1 in \cite{FW} to conclude that the function $g$ 
is uniformly approximable on $\overline{\D}^2$ by functions in $PSH(\D^2)\cap \mathcal{C}^{\infty}(\overline{\D}^2)$
since $b\D^2$ is not smooth. However, this is obvious since $\D^2$ is starshaped; $g$ can even be approximated uniformly
on $\overline{\D}^2$ by functions pluriharmonic in a neighborhood of $\overline{\D}^2$.   

Assume there is no holomorphic disk in $\overline{\D}^2$ where all $h_j=h_j(z,w)$ are holomorphic.
We will show that the polynomials in $\C^2\times \C^N$ are dense in $\mathcal{C}(\mathcal{G}_h(\overline{\D}^2))$.
 
The part of the boundary $b\D^2\setminus \Gamma^2$ is the disjoint union 
$(\D \times S^1) \cup (S^1\times \D)$.  
Consider the part $\D \times S^1\subset b\D^2\setminus \Gamma^2$; the other part is treated in a completely analogous way.
Let $g_j$ be the complex conjugate of the the restriction of $\partial h_j/\partial \bar{z}$ to $\D\times S^1$.
Then $g_j(z,s)\in \mathcal{C}(\overline{\D}\times S^1)$ is 
holomorphic in $\D$ for each fixed $s\in S^1$. Moreover, by expressing $h_j(\cdot,s)$ 
as a Poisson integral and differentiating under the integral sign, it follows that 
$g_j\in \mathcal{C}^1(\D\times S^1)$. Let 
\begin{equation*}
Z=\{(z,s)\in \D\times S^1;\, g_j(z,s)=0,\, \forall j\}.
\end{equation*}
Then, by Corollary~\ref{dbarCor}, the graph of $h$ over $(\D\times S^1) \setminus Z$ is a totally real manifold.

\begin{lemma}\label{cover}
There are closed sets $B\subset E\subset \D\times S^1$ such that
\begin{itemize}
\item[(i)] $Z\subset E$
\item[(ii)] $E\setminus B$ is a $\mathcal{C}^1$-smooth manifold of real dimension $1$
\item[(iii)] for any $\delta>0$, one can cover $B$ by the union of finitely many pairwise disjoint open sets 
$U_0,\ldots,U_k$ such that $\mbox{diam}\,(U_j)<\delta$, $j=1,\ldots,k$ and $U_0\subset \{(z,s)\in \D\times S^1;\, |z|>1-\delta\}$.
\end{itemize}
\end{lemma}

We take this lemma for granted for the moment and finish the proof of Theorem~\ref{polydisk}.
We define a stratification $Y_{-3}\subset \cdots \subset Y_2=\overline{\D}^2$; cf., Section~\ref{generalsektion}.  
Let 
\begin{eqnarray*}
Y_2 &=& \overline{\D}^2,\\
Y_1 &=& b\D^2 \cup \{z\in \D^2;\, H_2(z)=0\}=:b\D^2 \cup Z_1,\\
Y_0 &=& b\D^2 \cup \textrm{Sing}(Z_{1}) \cup \textrm{Reg}(Z_{1})^1_{H},\\
Y_{-1} &=& b\D^2, \\
Y_{-2} &=& \Gamma^2 \cup E\cup E', \\
Y_{-3} &=& \Gamma^2 \cup B\cup B',
\end{eqnarray*}
where $B'\subset E' \subset S^1\times \D \subset b\D^2$ are analogous to $B\subset E$; see Subsection~\ref{pharm}
for the definition of $\textrm{Reg}(Z_{1})^1_{H}$.
Since there is no holomorphic disk in $\overline{\D}^2$ where all $h_j$ are holomorphic, it follows from 
Proposition~\ref{holodisk} that $\textrm{dim}\, (Y_j\cap \D^2)=j$ for $j=0,1,2$, and by Corollary~\ref{dbarCor} 
it follows that
the graph of $h$ over $Y_j\setminus Y_{j-1}$ is totally real, $j=1,2$.
Notice also that each $Y_j$ is closed. It now suffices to show that
\begin{equation*}
X_{0}:=\cG_h(Y_{-3}) \subset \cdots \subset X_5:=\cG_h(Y_2)
\end{equation*}
fulfills the requirements of Theorem~\ref{approx}. We have seen that $X_5$ is polynomially convex and that
$X_{j}\setminus X_{j-1}$ is a totally real manifold for $j \geq 2$. However, $X_1\setminus X_0$ is the graph of $h$ over
$Y_{-2}\setminus Y_{-3}=(E\cup E')\setminus(B\cup B')$ which, by Lemma~\ref{cover}, is a $\mathcal{C}^1$-smooth manifold of real dimension $1$.
Hence, $X_1\setminus X_0$ is also totally real. It remains to see that $\mathcal{C}(X_0)=\hol(X_0)$.
Since $X_0$ is the graph of $h$ over $Y_{-3}$, it suffices to show that $\mathcal{C}(Y_{-3})=\hol(Y_{-3})$.
Let $\varphi \in \mathcal{C}(Y_{-3})$, let $\epsilon>0$, and let $\tilde{f}\in \hol(\Gamma^2)$ 
be such that $|\varphi -\tilde{f}|< \epsilon/2$ on $\Gamma^2$. 
From Lemma~\ref{cover} it follows that we, for any $\delta>0$, can cover
$Y_{-3}$ by the union of disjoint open sets $U_0,\ldots,U_{\ell}$ such that $\mbox{diam}\, (U_j)<\delta$, $j=1,\ldots,\ell$
and $\sup_{x\in U_0} \mbox{dist}\, (x,\Gamma^2)<\delta$. If $\delta$ is sufficiently small it thus follows that 
$\tilde{f}\in \hol(U_0)$ and $|\varphi -\tilde{f}|< \epsilon$ on $U_0 \cap Y_{-3}$.
Moreover, perhaps after shrinking $\delta$, there are constant functions $c_j$ that 
satisfies $|c_j-\varphi|<\epsilon$ in $U_j\cap Y_{-3}$, $j=1,\ldots,\ell$. 
We define $f$ to be equal to $\tilde{f}$ in $U_0$ and $c_j$ on $U_j$, $j=1,\ldots,\ell$.
Then $f\in \hol(Y_{-3})$ and $|f-\varphi|<\epsilon$.

\bigskip

It remains to prove Lemma~\ref{cover}.
For fixed $s_0\in S^1$ there is a $j$ such that $g_j(\cdot,s_0)$ does not vanish identically since 
there is no holomorphic disk in $\overline{\D}^2$ where all $h_j$ are holomorphic.
Hence, there is a neighborhood $I_0$ of $s_0$ in $S^1$ such that $g_j(\cdot,s)$ does not vanish identically for $s\in I_0$.
We can thus find connected pairwise disjoint open $I_1,\ldots,I_{\ell}\subset S^1$ such that $S^1=\cup_j \bar{I}_j$
and for each $i=1,\ldots,\ell$ there is a $g_{j_i}$ such that $g_{j_i}(\cdot,s)$ does not vanish identically 
for fixed $s\in \bar{I}_i$.
We let 
\begin{equation*}
E_i=\{(z,s)\in \D\times \bar{I}_i;\, g_{j_i}(z,s)=0\}, \quad E=\cup_{i=1}^{\ell}E_i;
\end{equation*} 
clearly $Z\subset E$.
We then let $B_i$ be the union of $E_i\cap (\D\times bI_i)$ and the set of points
$(z,s)\in E_i\cap (\D\times I_i)$ such that for every neighborhood $V\ni (z,s)$, $E_i\cap V$ is not a $\mathcal{C}^1$-smooth manifold
of real dimension $1$. Letting $B=\cup_{i=1}^{\ell} B_i$ it follows that $B\subset E$ is closed and that
$E\setminus B$ is a $\mathcal{C}^1$-smooth manifold of real dimension $1$.
Parts (i) and (ii) are proved.

Let $A=\{(z,s)\in B;\, |z|\leq 1-\delta\}$. To prove part (iii) it suffices to cover $A$ by the union of open sets
$U'_1,\ldots,U'_k$ such that $\mbox{diam}\,(U'_j)<\delta/2$ and $B\cap (\cup_j bU'_j)=\emptyset$.
In fact, then we can take $U_1=U'_1$, $U_2=U'_2\setminus \overline{U}_1$, $U_3=U'_3\setminus (\overline{U}_1\cup \overline{U}_2)$, and so on;
finally we take $U_0=\{(z,s)\in \D\times S^1;\, |z|>1-\delta\}\setminus (\cup_{j=1}^k\overline{U}_j)$.

Fix $s_0\in S^1$. Then there is a $j$ such that $s_0\in \bar{I}_j$ and if $s_0\in I_j$ then $s_0$ does not belong to any other 
$I_i$. Assume first that $s_0\in I_1$. Then $g_{j_1}(\cdot,s_0)$ does not vanish identically
and we let $\{g_{j_1}(\cdot,s_0)=0\}\cap \{|z|\leq 1-\delta\}=\{a_1(s_0),\ldots,a_m(s_0)\}$. Let $V_j\subset \D$, $j=1,\ldots,m$,
be pairwise disjoint open neighborhoods of $a_j(s_0)$ such that $\mbox{diam}\,(V_j)<\delta/10$ and $g_{j_1}(\cdot,s_0)|_{bV_j}\neq 0$.
By Lemma~\ref{known} below, there is a neighborhood $I_0\subset I_1$ of $s_0$ such that $\mbox{diam}\,(I_0)<\delta/10$ and such that
$a_j(s)$, $j=1,\ldots,m$, is $\mathcal{C}^1$-smooth in a neighborhood of $bI_0$ and $g_{j_1}(\cdot,s)|_{bV_j}\neq 0$ for $s\in \bar{I}_0$.
Letting $U'_j=V_j\times I_0$ we see that $A\cap (\D\times I_0)$ is covered by by the union of the $U'_j$, that $\mbox{diam}\,(U'_j)<\delta/2$,
and that $\cup_j bU'_j \cap B=\emptyset$. 
If instead $s_0\in bI_1$ then there is a unique $j\neq 1$ such that also $s_0\in bI_j$; say that $s_0\in bI_1\cap bI_2$. Then neither
$g_{j_1}(\cdot,s_0)$ nor $g_{j_2}(\cdot,s_0)$ vanishes identically and we can use the product $g_{j_1}\cdot g_{j_2}$ in the above construction
to find a neighborhood $I_0\subset S^1$ of $s_0$ and finitely many $U'_j$ covering $A\cap (\D\times I_0)$.
By compactness of $S^1$ we find the desired covering of $A$.

\begin{lemma}\label{known}
Let $g\in \mathcal{C}^1(\D\times (-1,1))$. Assume that $g(\cdot,t)$ is holomorphic for each fixed $t\in (-1,1)$ and that 
$0$ is an isolated zero of $g(\cdot,0)$. Let $V$ be a neighborhood of $0$ not containing any other zero of $g(\cdot,0)$
and assume that $g(\cdot,0)|_{bV}\neq 0$. 
Then there is an $\epsilon>0$ and a closed subset $K\subset(-\epsilon,\epsilon)$ without interior such that
$\{g(\cdot,t)=0\}\cap V=\{a_1(t),\ldots,a_m(t)\}$ for $t\in (-\epsilon,\epsilon)$ and all $a_j(t)$ are 
$\mathcal{C}^1$-smooth in $(-\epsilon,\epsilon)\setminus K$.
\end{lemma}  

\begin{proof}
This lemma should be well known so we only sketch a proof.
Let $g'(z,t)$ denote the derivative of $g$ with respect to $z$; by the Cauchy integral formula it follows that 
$g'(z,t)\in \mathcal{C}^1(\D\times (-1,1))$. The mapping 
\begin{equation}\label{index}
t\mapsto \frac{1}{2\pi i}\int_{bV} \frac{g'(z,t)}{g(z,t)}\, dz
\end{equation}
is continuous for $t$ close to $0$ and takes values in $\N$; thus it is constant. If it is $1$ it follows that
$g(\cdot,t)$ has a simple zero, $a(t)$, in $V$ for small fixed $t$. Then, by the residue theorem, it follows that
\begin{equation}\label{res1}
t\mapsto \frac{1}{2\pi i}\int_{bV} z\,\frac{g'(z,t)}{g(z,t)}\, dz
\end{equation} 
is equal to $a(t)$. Differentiating under the integral sign we see that $a(t)$ is $\mathcal{C}^1$-smooth for $t$ close to $0$.

If the mapping \eqref{index} equals $2$, then $g(\cdot,t)$ has two zeroes, $a_1(t), a_2(t)$, possibly coinciding, 
in $V$ for small fixed $t$. The mapping \eqref{res1} now equals $a_1(t)+a_2(t)$ and it is still $\mathcal{C}^1$-smooth.
We say that $t_0$ is {\em branching} if $a_1(t_0)=a_2(t_0)$ and if there for every $\delta>0$ exists a $t$ such that 
$|t-t_0|<\delta$ and $a_1(t)\neq a_2(t)$. Let $K$ be the set of branching $t$'s. More formally, one can define $K$ 
as the boundary of the zero set of the $\mathcal{C}^1$-mapping
\begin{equation}\label{res2}
t\mapsto \frac{2}{2\pi i}\int_{bV} z^2\,\frac{g'(z,t)}{g(z,t)}\, dz -
\left(\frac{1}{2\pi i}\int_{bV} z\,\frac{g'(z,t)}{g(z,t)}\, dz\right)^2 =
(a_1(t)-a_2(t))^2.
\end{equation}
Then it is clear that $K$ is closed and without interior. Let $t_0$ be a point outside $K$. Then either \eqref{res2}
is non-zero in a neighborhood of $t_0$ or \eqref{res2} is $0$ in a neighborhood of $t_0$. In the first case both 
$a_1(t)$ and $a_2(t)$ are simple zeroes of $g(\cdot,t)$ for fixed $t$ close to $t_0$ and it follows from the first part of 
the proof that both $a_1(t)$ and $a_2(t)$ are $\mathcal{C}^1$-smooth close to $t_0$. In the second case we have that
$a_1(t)=a_2(t)$ for $t$ close to $t_0$ and so the $\mathcal{C}^1$-mapping \eqref{res1} is equal to $2a_1(t)=2a_2(t)$
close to $t_0$.

The case when \eqref{index} equals $m>2$ is treated similarly.
\end{proof}

\section{Proof of Theorem \ref{maximality}}

By a result of Tornehave (see \emph{e.g.} Corollary 3.8.11 in Stout \cite{Stout})
it is enough to prove that there exists such an \emph{analytic} set $Z$. \

Let $h_j$ denote the pluriharmonic extension of $f_j$ to $\mathbb D^2$, and 
write $h=(h_1,...,h_N):\overline{\mathbb D^2}\rightarrow\mathbb C^N$.  
We let $\mathcal G_{h}(\overline{\mathbb D^2})$ denote the graph of $h$ over 
$\overline{\mathbb D^2}$ in $\mathbb C^2\times\mathbb C^N$, and 
we let $\mathcal G_{h}(\Gamma^2)$ denote the graph over $\Gamma^2$. 
Since $\Gamma^2$ is totally real it suffices to show that \emph{either $\mathcal G_h(\Gamma^2)$
is polynomially convex, or there exists a variety $Z\subset\overline{\mathbb D^2}\setminus \Gamma^2$ 
with $\overline Z\setminus Z\subset\Gamma^2$, all $h_j$-s holomorphic on $Z$}, 
and $\mathcal G_h(Z)\subset\widehat{\mathcal G_h(\Gamma^2)}$.
We assume that $\mathcal G_h(\Gamma^2)$ is not polynomially convex, and proceed to find a variety $Z$.  \
We will consider different possibilities through some lemmas, and then we will sum up the entire argument in the end.

The first thing we want to show is that 
\begin{lemma}\label{diskinbd}
Either $\widehat{\mathcal G_h(\Gamma^2)}$ contains a holomorphic disk $\mathcal G_h(\triangle)$,
where $\triangle$ is one of the disks in $b\mathbb D^2\setminus\Gamma^2$, 
or $\widehat{\mathcal G_h(\Gamma^2)}\cap (\mathcal G_h(b\mathbb D^2\setminus\Gamma^2))=\emptyset$. 
\end{lemma}
\begin{proof}
Choose a point $\zeta=(\zeta_1,\zeta_2)\in b\mathbb D^2\setminus\Gamma^2$ with $\mathcal G_h(\zeta)\in\widehat{\mathcal G_h(\Gamma^2)}$.  
Without loss of generality we assume that $|\zeta_1|=1$.  Let $\triangle$
be the disk $\triangle:=\{(\zeta_1,w):|w|< 1\}$.  Since 
each point of $b\mathbb D_{z_1}$ is a peak point for $\mathcal O(\mathbb C_{z_1})|_{\overline{\mathbb D}}$ it 
follows that $\mathcal G_h(\zeta)\subset\widehat{\mathcal G_h(b\triangle)}$.
By Wermer \cite{Wermer65} we have that $h_j$ is holomorphic on $\triangle$
for $j=1,...,N$.  \
\end{proof}

As we proceed we assume that there is no such disk $\triangle$, \emph{i.e.}, if we can locate 
a closed variety in $\widehat{\mathcal G_h(\Gamma^2)}\cap\mathcal G_h(\mathbb D^2)$, 
then it is automatically attached to $\mathcal G_h(\Gamma^2)$.  \

Let 
$$
\tilde Z:=\{z\in{\mathbb D^2}:\overline\partial h_{i_1}\wedge\overline\partial h_{i_2}(z)=0, \forall (i_1,i_2), 1\leq i_1,i_2\leq N \}.
$$
By Lemma \ref{pcgraph} we have that $\mathcal G_h(\overline{\mathbb D^2})$ is polynomially convex, and since 
$\mathcal G_h({\mathbb D^2}\setminus\tilde Z)$ is totally real by Lemma \ref{totreellLemma}, it 
follows from Theorem \ref{approx} that 
$$
\widehat{\mathcal G_h(\Gamma^2)} \subset\mathcal G_h(\Gamma^2\cup\tilde Z).
$$
(Because all totally real points are peak-points.)

\begin{lemma}\label{jacnotzero}
Assume that $\tilde Z\neq {\mathbb D^2}$, let $Z_\alpha$ be an irreducible component
of $\tilde Z$of dimension 1, and let $z_0\in Z_{\alpha}\cap\tilde Z_{reg}$ with $\mathcal G_h(z_0)\in\widehat{\mathcal G_h(\Gamma^2)}$.
Then $\mathcal G_h(Z_\alpha)\subset \widehat{\mathcal G_h(\Gamma^2)}$.  All the $h_j$-s 
are holomorphic along $Z_\alpha$.
\end{lemma}

\begin{proof}
Let $\Omega$ be a small open neighborhood of $z_0$ such that $\Omega\cap Z_\alpha$ 
is a smooth disk $D_{z_0}$. 
Then, if $\Omega$ is small enough, we have that $K:=(b\Omega\times\mathbb C^N)\cap\widehat{\mathcal G_h(\Gamma^2)}\subset\mathcal G_h(bD_{z_0})$, and so by Rossi's local maximum principle 
$$
\mathcal G_h(z_0)\in\widehat K\subset\widehat{\mathcal G_h(bD_{z_0})}.
$$  
By Wermer's maximality theorem it follows that all $h_j$-s are holomorphic 
on $Z_\alpha$ near $z_0$, and then $\mathcal G_h(D_{z_0})$ is contained in the hull,
since we must have that  $K=\mathcal G_h(bD_{z_0})$.  
Since this holds near any point of $Z_\alpha\cap\tilde Z_{reg}$ in the hull, 
it follows that 
$Z_\alpha\cap\tilde Z_{reg}$ is contained in the hull, hence $Z_\alpha$ is contained in the hull.  \

\end{proof}

\medskip

Finally we need to consider the case that $\tilde Z={\mathbb D^2}$.
We want to change coordinates.  For each $j$ let $g_j\in\mathcal O(\mathbb D^2)$
with $Im(g_j)=Im(h_j)$, and let $\varphi_j\in\mathcal O(\mathbb D^2)$
with $Re(\varphi_j)=u_j:=h_j-g_j$.  Then $u_j$ is real for all $j$, 
and for any compact set $K\subset\mathbb D^2$ we have that
$\mathcal G_h(K)$ 
being polynomially convex is equivalent to $\mathcal G_u(K)$ being polynomially convex.  
We will show that $\widehat{\mathcal G_h(\Gamma^2)}$ contains the graph of a leaf of the (possibly singular) Levi-foliation of a level set $\{u_j=c\}$ for at least one $j$.    \
We want to use our coordinate change to study the hull, but since the $u_j$-s are
not necessarily continuous up to $\Gamma^2$ we will consider a certain exhaustion 
of $\mathbb D^2$.  \

By the assumption that $\widehat{\mathcal G_h(\Gamma^2)}\cap (\mathcal G_h(b\mathbb D^2\setminus\Gamma^2))=\emptyset$ there exist sequences of 
real numbers $r_j\rightarrow 1, \epsilon_j\rightarrow 0$, such that the following holds: defining 
\begin{align*}
Q_j :&= \{(z_1,z_2):|z_1|=r_j, r_j-\epsilon_j\leq |z_2|\leq r_j\}\\
 & Ê\cup \{(z_1,z_2):|z_2|=r_j, r_j-\epsilon_j\leq |z_1|\leq r_j\}
\end{align*}
we have that 
\begin{itemize}
\item[1)]$\tilde K_j:=\widehat{\mathcal G_h(\Gamma^2)}\cap \mathcal G_h(b\mathbb D_{r_j}^2)\subset\mathcal G_h(Q_j)$.
\end{itemize}

Let $K_j$ denote the projection of $\tilde K_j$ to $\mathbb C^2$.  Then the $h_j$-s, 
and consequently the $u_i$-s, are smooth in a neighborhood of $K_j$ for all $j$. 

\medskip

\begin{lemma}
Assume that there is a subsequence $r_{j_k}$ of $r_j$ such that 
for each $k$ there exists a point $\mathcal G_h((a_k,b_k))\in\widehat{\mathcal G_h(Q_{j_k})}$ with 
$|a_k|=r_{k_{j_k}}$ and $|b_k|<r_{j_k}-\epsilon_{j_k}$.
Then there exists a disk $\mathbb D_a=\{a\}\times\mathbb D$ with $|a|=1$, 
and all $h_j$-s holomorphic on $\mathbb D_a$.
\end{lemma}
\begin{proof}
This is similar to the proof of Lemma \ref{diskinbd}. 
Since each point $a_k$ is a peak point for $\mathcal O(\overline{\mathbb D_{|a_k|}})$
it follows from \cite{AxlerShields} (alternatively Wermer's Maximality Theorem combined with Theorem \ref{domain}) that all $h_j$-s are 
holomorphic on $\mathbb D_{a_k}$.  By passing to 
a subsequence we may assume that $a_k\rightarrow a\in b\mathbb D$
as $k\rightarrow\infty$.  
\end{proof}

We will now make the assumption that 
\begin{itemize}
\item[2)] none of the functions $h_j$ are holomorphic on any of the disks
in $b\mathbb D^2_{r_j}\setminus Q_j$, or, equivalently, for any point 
$x\in b\mathbb D^2_{r_j}\setminus Q_j$  we have that 
$\mathcal G_h(x)\notin\widehat{\mathcal G_h(Q_j)}$. 
\end{itemize}

We will now consider the case that 
$\tilde Z={\mathbb D^2}$ and that at least one of the $u_j$-s, say $u_1$, is non-constant.  
Define $L_c=L_c(u_1):=\{z\in\mathbb D^2: u_1(z)=c\}$.

We have that 
$$
L_c(u_1)=\underset{r\in\mathbb R}{\bigcup}L_{c+i\cdot r}(\varphi_1),
$$
so $L_c$ is the disjoint union of analytic sets, which we will call 
leaves of the (singular) lamination $L_c$.

\begin{lemma}
There exists a discrete set $A\subset\mathbb D^2$
such that near any point $z_0\in\mathbb D^2\setminus A$, the 
set $\{\varphi_1(z)=\varphi_1(z_0)\}$ is a smooth surface.  
\end{lemma}

\begin{proof}
Let 
$$
W=\{\partial\varphi_1=0\}.
$$
Assume for simplicity that $W$ is a connected 1-dimensional variety
with $\varphi_1|_W\equiv 0$.  
Let $\tilde W:=\{\varphi_1=0\}$.  Then $A=\tilde W_{sing}\cap W$
is a discrete set.  
\end{proof}

\begin{lemma}\label{lamination}
Assume that $\tilde Z={\mathbb D^2}$.
Assume that $\mathcal G_{h}(x_0)\in\widehat{\mathcal G_{h}(Q_j)}\cap\mathcal G_{h}(\mathbb D_{r_j})$ and set $c_1=u_1(x_0)$.   
Then there exists a point $z_0\in (L_{c_1}\setminus A)\cap\mathbb D^2_{r_j}$
with $\mathcal G_{h}(z_0)\in\widehat{\mathcal G_{h}(Q_j)}$.  Moreover, 
for any point $z_0\in L_{c_1}\setminus A$, we have that 
$\mathcal G_h(z_0)\in\widehat{\mathcal G_h(Q_j)}$ implies 
that $\mathcal G_h(\mathcal L_{z_0}\cap\mathbb D^2_{r_j})\subset\widehat{\mathcal G_h(Q_j)}$
with $h$ holomorphic along $\mathcal L_{z_0}$, where $\mathcal L_{z_0}$
denotes the leaf through $z_0$.

\end{lemma}

\begin{proof}

We have that $\mathcal G_h(z_0)\in\widehat{\mathcal G_h(Q_j)}$
is equivalent to $\mathcal G_u(z_0)\in\widehat{\mathcal G_u(Q_j)}$.

Let $c=u(z_0)$ and  $Q_j^c:=\{z\in Q_j: h(z)=c\}$.
By Proposition \ref{easybishop} we have that 
$$
\mathcal G_{u}(z_0)\in\widehat{\mathcal G_{u}(Q_j^c)}\subset L_c.
$$
Since $Q_j^c$ is a level set of $u$ this means that $z_0\in\widehat{Q_j^c}$.
It also follows that the hull contains graphs over the regular points, because 
the singular set is discrete.

\begin{lemma}\label{leafinhull}
Let $z_0\in L_{c_1}\cap (\mathbb D_{r_j}^2\setminus A)$ and let $\mathcal L_{z_0}$ denote the leaf through $z_0$.
Then $z_0\in\widehat{Q_j^c}$ if and only if $\mathcal L_{z_0}\subset\widehat{Q_j^c}$.
\end{lemma}
\begin{proof}
This is quite similar to the proof of Lemma \ref{jacnotzero}.
Choose a small neighborhood $\Omega$ of $z_0$ such that $D_{z_0}=\Omega\cap\mathcal L_{z_0}$
is a smoothly bounded disk.  We may assume that $\varphi_1(z_0)=0$ and 
$\varphi_1((L_c\cap\overline{\Omega})\setminus D_{z_0})\subset i\cdot\mathbb R\setminus\{0\}$. \

Let $K:=\widehat{Q_j^c}\cap b\Omega$.  Then $z_0\in\widehat{K}$ by Rossi's principle.  
This implies that $bD_{z_0}\subset K$.   Otherwise, let $K_{z_0}:=bD_{z_0}\cap K$
and let $g$ be holomorphic with $|g(z_0)|>\|g\|_{K_{z_0}}$.  Let $\tau$
be holomorphic on $\mathbb C$, $\tau(0)=1$, $0<\tau(z)<1$ for $z\in i\cdot\mathbb R\setminus\{0\}$. 
Then $g\cdot (\tau\circ\varphi_1)^m$ separates $z_0$ from $K$ for large $m$.  This is a
contradiction. \

Then clearly $D_{z_0}$ is contained in the hull, and since $K_{z_0}=bD_{z_0}$
we have that  $u_2,...,u_N$
are also constant on $D_{z_0}$.

\end{proof}

Finally note that if $\mathcal L_{z_0}\subset\widehat{Q^c_j}$, 
then $\mathcal G_u(\mathcal L_{z_0})\subset \widehat{\mathcal G_u(Q^c_j)}$.

\end{proof}

\emph{Proof of Theorem \ref{maximality}:}

Let $h_j$ be the pluriharmonic extension of $f_j$ to $\mathbb D^2$ for $j=1,...,N$.
Clearly we may assume that not all $h_j$-s are constant.    
If all $h_j$-s are holomorphic on a disk $\triangle$ in $b\mathbb D^2\setminus\Gamma^2$
the algebra is clearly not generated, and the conclusion of the theorem would hold.  
So from now on we assume that not all of the functions are holomorphic any of these disks. 
By Lemma \ref{diskinbd} none of these disks intersect the hull.  

\medskip

Consider 
$$
\tilde Z:=\{z\in {\mathbb D^n}:\overline\partial h_{i_1}\wedge\overline\partial h_{i_2}(z)=0, \forall (i_1,i_2), 1\leq i_1,i_2\leq N \}.
$$

Assume that $\tilde Z\neq {\mathbb D^2}$.   If $\mathcal G_h(\Gamma^2)$ is not polynomially convex, then according to Lemma \ref{jacnotzero} there is an irreducible component $Z_\alpha$ of $\tilde Z$
on which all $h_j$-s are holomorphic in the hull.   By assumption $\overline Z_\alpha\cap b\mathbb D^2\subset\Gamma^2$.

\medskip

Finally we consider the case that $\tilde Z={\mathbb D^2}$.  
We assume that $h_1$ is non-holomorphic.  
Let $\mathcal G_h(x_0)\in\widehat{\mathcal G_h(\Gamma^2)}$. 
For $j$ large enough such that $x_0\in\mathbb D^2_{r_j}$ 
we then have that $\mathcal G_h(x_0)\in\widehat{\mathcal G_h(Q_j)}$
(by 1) and Rossi's principle).  \

By Lemma \ref{lamination} we may assume that $x_0\notin A$.  
Assume that $\varphi_1(x_0)=0$ and let $Z$ denote the 
irreducible component of $\{z\in\mathbb D^2:\varphi_1(z)=0\}$
containing $x_0$.  According to Lemma \ref{lamination}
we have that $\mathcal G_h(Z\cap\mathbb D^2_{r_j})\subset \widehat{\cG_h(Q_j)}$
for all $j$.   It follows that $G_h(Z)\subset \widehat{\cG_h(\Gamma^2)}$.

$\hfill\square$

\medskip

\begin{example}\label{example}
Let $f(z_1,z_2):=Re(z_1 + c\cdot z_2)$ with $c\in\mathbb C^*$.  
Since $f$ is real and pluriharmonic on $\mathbb C^2$ we set $h=f$. 
We will show that $[z_1,z_2,f]_{\Gamma^2}=\mathcal C(\Gamma^2)$.
Otherwise, by Theorem \ref{maximality}, there is a non-trivial closed variety
$Z\subset\mathbb D^2$ with $\overline Z\setminus Z\subset\Gamma^2$ on which $h$ is holomorphic.  Then $h$ is constant 
$Z$, so $Z$ is contained in a hypersurface
$$
\{z_1 + c\cdot z_2 = r + i\cdot s: r \mbox{ fixed and } s\in\mathbb R\}.
$$
But then $Z$ is is contained in a complex line 
$$
L_k:=\{z_1 + c\cdot z_2 = k: k\in\mathbb C\}.
$$
Since $L_k\cap b\mathbb D^2$ is not contained in $\Gamma^2$ this is impossible. 

\end{example}

\section{Proof of Theorem \ref{domain}}\label{generalsektion}
Assume that there is no analytic disk in $\Omega$ where all the $h_j$ are holomorphic. We will prove that 
the polynomials in $\C^n\times \C^N$ are dense in the uniform algebra of continuous functions on the graph 
$\cG_h(\overline{\Omega})$. By Lemma \ref{pcgraph} we have that $\cG_h(\overline{\Omega})$ is polynomially convex.

\smallskip

We will define a stratification of $\cG_h(\overline{\Omega})$ so that we can use Theorem~\ref{approx}. 
As in Section~\ref{bidisksektion} we begin by stratifying $\overline{\Omega}$:
Let 
\begin{eqnarray*}
Y_n &=& \overline{\Omega}, \\
Y_{n-1} &=& b\Omega \cup \Omega^n_H=:b\Omega \cup Z_{n-1};
\end{eqnarray*}
see Subsection~\ref{pharm} for the definition of $\Omega^n_H$. 
Since there is no analytic disk in $\Omega$ where all $h_j$ are holomorphic, 
it follows from Proposition~\ref{holodisk} that $\mbox{dim}\, Z_{n-1} \leq n-1$. 
For an analytic set $V$ we write $V'$ for the union of the irreducible components of $V$ of maximal dimension 
and $V''$ for the union of the rest of the components. Let
\begin{eqnarray*}
Y_{n-2} &=& b\Omega \cup \textrm{Sing}(Z_{n-1}) \cup Z''_{n-1} \cup \textrm{Reg}(Z'_{n-1})^{n-1}_H \\
&=:& b\Omega \cup Z_{n-2}.
\end{eqnarray*}
Again, from Proposition~\ref{holodisk} it follows that $\mbox{dim}\, Z_{n-2}\leq n-2$; notice that $\mbox{dim}\, (Z_{n-1})_{sing}\leq n-2$ and 
$\mbox{dim}\, Z''_{n-1}\leq n-2$ automatically. We define recursivly
\begin{eqnarray*}
Y_{k} &=& b\Omega \cup \textrm{Sing}(Z_{k+1}) \cup Z''_{k+1} \cup \textrm{Reg}(Z'_{k+1})^{k+1}_H \\
&=:& b\Omega \cup Z_{k}
\end{eqnarray*}
and from Proposition~\ref{holodisk} we have $\mbox{dim}\, Z_k \leq k$. Moreover, $Y_k\setminus Y_{k-1}$ is (either empty or)
a $k$-dimensional complex manifold where $H_k$ is non-vanishing, and so, by Corollary~\ref{dbarCor}, 
the graph of $h|_{Y_k\setminus Y_{k-1}}$ is a totally real manifold. We define our stratification of $\cG_h(\overline{\Omega})$ as follows:
\begin{equation*}
X_0:=\cG_h(Y_0) \subset \cdots \subset X_n:=\cG_h(Y_n).
\end{equation*}
We have seen that $X_n$ is polynomially convex, that $X_k\setminus X_{k-1}$ is a totally real manifold, and that 
$X_0$ is the union of the image of $b\Omega$ and a discrete set in $\Omega$. Since by assumption 
$[z_1,\ldots,z_n,h_1,\ldots,h_N]_{b\Omega} = \cC(b\Omega)$ 
it follows that $\hol(X_0)$ is dense in $\cC(X_0)$. By Theorem~\ref{approx} we are thus done.

\bibliographystyle{amsplain}

\end{document}